\documentclass{article}
\usepackage{amsmath}
\usepackage{amssymb}
\usepackage{amsfonts}
\usepackage{hyperref}

\newtheorem{theorem}{Theorem}

\newtheorem{corollary}[theorem]{Corollary}

\newtheorem{proposition}[theorem]{Proposition}
\newtheorem{remark}[theorem]{Remark}

\newenvironment{proof}[1][Proof]{\noindent\textbf{#1.} }{\ \rule{0.5em}{0.5em}}
\input{tcilatex}
\numberwithin{theorem}{section}
\numberwithin{equation}{section}

\begin{document}

\title{A Relativistic Approach on $1$-Jet Spaces of the Rheonomic Berwald-Mo%
\'{o}r Metric}
\author{Mircea Neagu \\
{\scriptsize February 2010; Revised September 2010}\\
{\scriptsize (correction upon the canonical nonlinear connection)}}
\date{}
\maketitle

\begin{abstract}
The aim of this paper is to develop on the $1$-jet space $J^{1}(\mathbb{R}%
,M^{4})$ the Finsler-like geometry (in the sense of d-connection, d-torsions
and d-curvatures) of the rheonomic Berwald-Mo\'{o}r metric 
\begin{equation*}
\mathring{F}(t,y)=\sqrt{h^{11}(t)}\sqrt[4]{%
y_{1}^{1}y_{1}^{2}y_{1}^{3}y_{1}^{4}}.
\end{equation*}%
A natural geometrical gravitational field theory produced by the above
rheonomic Berwald-Mo\'{o}r metric is also constructed.
\end{abstract}

\textbf{Mathematics Subject Classification (2000):} 53C60, 53C80, 83C22.

\textbf{Key words and phrases:} rheonomic Berwald-Mo\'{o}r metric, canonical
nonlinear connection, Cartan canonical connection, d-torsions and
d-curvatures, geometrical Einstein equations.

\section{Introduction}

\hspace{5mm}It is obvious that our natural physical intuition distinguishes
four dimensions in a natural correspondence with the material reality.
Consequently, the four dimensionality plays special role in almost all
modern physical theories.

On the other hand, it is an well known fact that, in order to create the
Relativity Theory, Einstein was forced to use the Riemannian geometry
instead of the classical Euclidean geometry, the first one representing the
natural ma\-the\-ma\-ti\-cal model for the local \textit{isotropic}
space-time. But, there are recent studies of physicists which suggest a 
\textit{non-isotropic} perspective of the space-time (for example, in
Pavlov's opinion \cite{Pavlov}, the concept of inertial body mass emphasizes
the necessity of study of local non-isotropic spaces). Obviously, for the
study of non-isotropic physical phenomena, the Finsler geometry is very
useful as ma\-the\-ma\-ti\-cal framework.

The studies of Russian scholars (Asanov \cite{Asanov[1]}, Mikhailov \cite%
{Mikhailov}, Garas'ko and Pavlov \cite{Garasko-Pavlov}) emphasize the
importance of the Finsler geometry which is characterized by the total
equality of all non-isotropic directions. For such a reason, Asanov, Pavlov
and their co-workers underline the important role played by the Berwald-Mo%
\'{o}r metric (whose Finsler geometry is deeply studied by Matsumoto and
Shimada in the paper \cite{Mats-Shimada})%
\begin{equation*}
F:TM\rightarrow \mathbb{R},\mathbb{\qquad }F(y)=\left(
y^{1}y^{2}...y^{n}\right) ^{\frac{1}{n}},
\end{equation*}%
in the theory of space-time structure and gravitation, as well as in unified
gauge field theories. Because any of such directions can be related to the
proper time of an inertial reference frame, Pavlov considers that it is
appropriate as such spaces to be generically called \textit{%
"multi-dimensional time"} \cite{Pavlov}. In the framework of the $4$%
-dimensional linear space with Berwald-Mo\'{o}r metric (i.e. the
four-dimensional time), Pavlov and his co-workers \cite{Garasko-Pavlov}, 
\cite{Pavlov} offer some new physical approaches and geometrical
interpretations such as:

1. physical events = points in the 4-dimensional space;

2. straight lines = shortest curves;

3. intervals = distances between the points along of a straight line;

4. light pyramids $\Leftrightarrow $ light cones in a pseudo-Euclidian space.

For such geometrical and physical reasons, this paper is devoted to the
development on the $1$-jet space $J^{1}(\mathbb{R},M^{4})$ of the
Finsler-like geometry (together with a theoretical-geometric gravitational
field theory) of the \textit{rheonomic Berwald-Mo\'{o}r metric }%
\begin{equation*}
\mathring{F}:J^{1}(\mathbb{R},M^{4})\rightarrow \mathbb{R},\qquad \mathring{F%
}(t,y)=\sqrt{h^{11}(t)}\sqrt[4]{y_{1}^{1}y_{1}^{2}y_{1}^{3}y_{1}^{4}},
\end{equation*}%
where $h_{11}(t)$ is a Riemannian metric on $\mathbb{R}$ and $%
(t,x^{1},x^{2},x^{3},x^{4},y_{1}^{1},y_{1}^{2},y_{1}^{3},y_{1}^{4})$ are the
coordinates of the $1$-jet space $J^{1}(\mathbb{R},M^{4})$.

The differential geometry (in the sense of distinguished (d-) connections,
d-torsions, d-curvatures, gravitational and electromagnetic geometrical
theories) produced by a jet rheonomic Lagrangian function $L:J^{1}(\mathbb{R}%
,M^{n})\rightarrow \mathbb{R}$ is now completely done in the author's paper 
\cite{Neagu-Rheon}. We point out that the geometrical ideas from \cite%
{Neagu-Rheon} are similar, but however distinct ones, with those exposed by
Miron and Anastasiei in the classical Lagrangian geometry \cite{Mir-An}. In
fact, the geometrical ideas from \cite{Neagu-Rheon} (the jet geometrical
theory of the \textit{rheonomic Lagrange spaces}) were initially stated by
Asanov in \cite{Asanov[2]} and developed further by the author of this paper
in the book \cite{Neagu Carte}.

In the sequel, we apply the general geometrical results from \cite%
{Neagu-Rheon} to the rheonomic Berwald-Mo\'{o}r metric\textit{\ }$\mathring{F%
}$.

\section{Preliminary notations and formulas}

\hspace{5mm}Let $(\mathbb{R},h_{11}(t))$ be a Riemannian manifold, where $%
\mathbb{R}$ is the set of real numbers. The Christoffel symbol of the
Riemannian metric $h_{11}(t)$ is%
\begin{equation*}
\varkappa _{11}^{1}=\frac{h^{11}}{2}\frac{dh_{11}}{dt},\qquad h^{11}=\frac{1%
}{h_{11}}>0.
\end{equation*}%
Let also $M^{4}$ be a manifold of dimension four, whose local coordinates
are $(x^{1},x^{2},x^{3},x^{4})$. Let us consider the $1$-jet space $J^{1}(%
\mathbb{R},M^{4})$, whose local coordinates are%
\begin{equation*}
(t,x^{1},x^{2},x^{3},x^{4},y_{1}^{1},y_{1}^{2},y_{1}^{3},y_{1}^{4}).
\end{equation*}%
These transform by the rules (the Einstein convention of summation is used
throughout this work):%
\begin{equation}
\widetilde{t}=\widetilde{t}(t),\quad \widetilde{x}^{p}=\widetilde{x}%
^{p}(x^{q}),\quad \widetilde{y}_{1}^{p}=\dfrac{\partial \widetilde{x}^{p}}{%
\partial x^{q}}\dfrac{dt}{d\widetilde{t}}\cdot y_{1}^{q},\qquad p,q=%
\overline{1,4},  \label{tr-rules}
\end{equation}%
where $d\widetilde{t}/dt\neq 0$ and rank $(\partial \widetilde{x}%
^{p}/\partial x^{q})=4$. We consider that the manifold $M^{4}$ is endowed
with a tensor of kind $(0,4)$, given by the local components $G_{pqrs}(x)$,
which is totally symmetric in the indices $p$, $q$, $r$ and $s$. Suppose
that the d-tensor 
\begin{equation*}
G_{ij11}=12G_{ijpq}y_{1}^{p}y_{1}^{q},
\end{equation*}%
is non-degenerate, that is there exists the d-tensor $G^{jk11}$ on $J^{1}(%
\mathbb{R},M^{4})$ such that $G_{ij11}G^{jk11}=\delta _{i}^{k}.$

In this geometrical context, if we use the notation $%
G_{1111}=G_{pqrs}y_{1}^{p}y_{1}^{q}y_{1}^{r}y_{1}^{s}$, we can consider the
rheonomic Finsler-like function (it is $1$-positive homogenous in the
variable $y$):%
\begin{equation}
F(t,x,y)=\sqrt[4]{G_{pqrs}(x)y_{1}^{p}y_{1}^{q}y_{1}^{r}y_{1}^{s}}\cdot 
\sqrt{h^{11}(t)}=\sqrt[4]{G_{1111}(x,y)}\cdot \sqrt{h^{11}(t)},  \label{F}
\end{equation}%
where the Finsler function $F$ has as domain of definition all values $%
(t,x,y)$ which verify the condition $G_{1111}(x,y)>0.$ If we denote $%
G_{i111}=4G_{ipqr}(x)y_{1}^{p}y_{1}^{q}y_{1}^{r}$, then the $4$-positive
homogeneity of the "$y$-function" $G_{1111}$ (this is in fact a d-tensor on $%
J^{1}(\mathbb{R},M^{4})$) leads to the equalities:%
\begin{equation*}
G_{i111}=\frac{\partial G_{1111}}{\partial y_{1}^{i}},\quad
G_{i111}y_{1}^{i}=4G_{1111},\quad G_{ij11}y_{1}^{j}=3G_{i111},
\end{equation*}%
\begin{equation*}
G_{ij11}=\frac{\partial G_{i111}}{\partial y_{1}^{j}}=\frac{\partial
^{2}G_{1111}}{\partial y_{1}^{i}\partial y_{1}^{j}},\quad
G_{ij11}y_{1}^{i}y_{1}^{j}=12G_{1111}.
\end{equation*}

The \textit{fundamental metrical d-tensor} produced by $F$ is given by the
formula%
\begin{equation*}
g_{ij}(t,x,y)=\frac{h_{11}(t)}{2}\frac{\partial ^{2}F^{2}}{\partial
y_{1}^{i}\partial y_{1}^{j}}.
\end{equation*}%
By direct computations, the fundamental metrical d-tensor takes the form%
\begin{equation}
g_{ij}(x,y)=\frac{1}{4\sqrt{G_{1111}}}\left[ G_{ij11}-\frac{1}{2G_{1111}}%
G_{i111}G_{j111}\right] .  \label{g-(ij)-general}
\end{equation}%
Moreover, taking into account that the d-tensor $G_{ij11}$ is
non-degenerate, we deduce that the matrix $g=(g_{ij})$ admits the inverse $%
g^{-1}=(g^{jk})$. The entries of the inverse matrix $g^{-1}$ are%
\begin{equation}
g^{jk}=4\sqrt{G_{1111}}\left[ G^{jk11}+\frac{G_{1}^{j}G_{1}^{k}}{2\left(
G_{1111}-\mathcal{G}_{1111}\right) }\right] ,  \label{g+(jk)-general}
\end{equation}%
where $G_{1}^{j}=G^{jp11}G_{p111}$ and $2\mathcal{G}%
_{1111}=G^{pq11}G_{p111}G_{q111}.$

\section{The rheonomic Berwald-Mo\'{o}r metric}

\hspace{5mm}Beginning with this Section we will focus only on the \textit{%
rheonomic Berwald-Mo\'{o}r metric}, which is the Finsler-like metric (\ref{F}%
) for the particular case%
\begin{equation*}
G_{pqrs}=\left\{ 
\begin{array}{ll}
\dfrac{1}{4!}, & \{p,q,r,s\}\text{ - distinct indices}\medskip \\ 
0, & \text{otherwise.}%
\end{array}%
\right.
\end{equation*}%
Consequently, the rheonomic Berwald-Mo\'{o}r metric is given by%
\begin{equation}
\mathring{F}(t,y)=\sqrt{h^{11}(t)}\cdot \sqrt[4]{%
y_{1}^{1}y_{1}^{2}y_{1}^{3}y_{1}^{4}}.  \label{rheon-B-M}
\end{equation}%
Moreover, using preceding notations and formulas, we obtain the following
relations:%
\begin{equation*}
G_{1111}=y_{1}^{1}y_{1}^{2}y_{1}^{3}y_{1}^{4},\quad G_{i111}=\frac{G_{1111}}{%
y_{1}^{i}},
\end{equation*}%
\begin{equation*}
G_{ij11}=\left( 1-\delta _{ij}\right) \frac{G_{1111}}{y_{1}^{i}y_{1}^{j}}%
\text{ (no sum by }i\text{ or }j\text{),}
\end{equation*}%
where $\delta _{ij}$ is the Kronecker symbol. Because we have%
\begin{equation*}
\det \left( G_{ij11}\right) _{i,j=\overline{1,4}}=-3\left( G_{1111}\right)
^{2}\neq 0,
\end{equation*}%
we find%
\begin{equation*}
G^{jk11}=\frac{(1-3\delta ^{jk})}{3G_{1111}}y_{1}^{j}y_{1}^{k}\text{ (no sum
by }j\text{ or }k\text{).}
\end{equation*}%
It follows that we have $\mathcal{G}_{1111}=(2/3)G_{1111}$ and $%
G_{1}^{j}=(1/3)y_{1}^{j}$.

Replacing now the preceding computed entities into the formulas (\ref%
{g-(ij)-general}) and (\ref{g+(jk)-general}), we get%
\begin{equation}
g_{ij}=\frac{\left( 1-2\delta _{ij}\right) \sqrt{G_{1111}}}{8}\frac{1}{%
y_{1}^{i}y_{1}^{j}}\text{ (no sum by }i\text{ or }j\text{)}
\label{g-jos-(ij)}
\end{equation}%
and%
\begin{equation}
g^{jk}=\frac{2(1-2\delta ^{jk})}{\sqrt{G_{1111}}}y_{1}^{j}y_{1}^{k}\text{
(no sum by }j\text{ or }k\text{).}  \label{g-sus-(jk)}
\end{equation}

Using a general formula from the paper \cite{Neagu-Rheon}, we find the
following geometrical result:

\begin{proposition}
For the rheonomic Berwald-Mo\'{o}r metric (\ref{rheon-B-M}), the \textit{%
energy action functional}%
\begin{equation*}
\mathbb{\mathring{E}}(t,x(t))=\int_{a}^{b}\sqrt{%
y_{1}^{1}y_{1}^{2}y_{1}^{3}y_{1}^{4}}\cdot h^{11}\sqrt{h_{11}}dt
\end{equation*}%
produces on the $1$-jet space $J^{1}(\mathbb{R},M^{4})$ the \textit{%
canonical nonlinear connection}%
\begin{equation}
\Gamma =\left( M_{(1)1}^{(i)}=-\varkappa _{11}^{1}y_{1}^{i},\text{ }%
N_{(1)j}^{(i)}=0\right) .  \label{can-nlc=0}
\end{equation}
\end{proposition}

Because the canonical nonlinear connection (\ref{can-nlc=0}) has the spatial
components equal to zero, it follows that our subsequent geometrical theory
becomes trivial, in a way. For such a reason, in order to avoid the
triviality of our theory and in order to have a certain kind of symmetry, we
will use on the $1$-jet space $J^{1}(\mathbb{R},M^{4})$, by an "a priori"
definition, the following nonlinear connection:

\begin{equation}
\mathring{\Gamma}=\left( M_{(1)1}^{(i)}=-\varkappa _{11}^{1}y_{1}^{i},\text{ 
}N_{(1)j}^{(i)}=-\frac{\varkappa _{11}^{1}}{3}\delta _{j}^{i}\right) .
\label{nlc-B-M}
\end{equation}

\section{Cartan canonical connection. d-Torsions and d-curvatures}

\hspace{5mm}The importance of the nonlinear connection (\ref{nlc-B-M}) is
coming from the possibility of construction of the dual \textit{adapted bases%
} of distinguished (d-) vector fields%
\begin{equation}
\left\{ \frac{\delta }{\delta t}=\frac{\partial }{\partial t}+\varkappa
_{11}^{1}y_{1}^{p}\frac{\partial }{\partial y_{1}^{p}},\text{ }\frac{\delta 
}{\delta x^{i}}=\frac{\partial }{\partial x^{i}}+\frac{\varkappa _{11}^{1}}{3%
}\frac{\partial }{\partial y_{1}^{i}},\text{ }\dfrac{\partial }{\partial
y_{1}^{i}}\right\} \subset \mathcal{X}(E)  \label{a-b-v}
\end{equation}%
and distinguished covector fields%
\begin{equation}
\left\{ dt,\text{ }dx^{i},\text{ }\delta y_{1}^{i}=dy_{1}^{i}-\varkappa
_{11}^{1}y_{1}^{i}dt-\frac{\varkappa _{11}^{1}}{3}dx^{i}\right\} \subset 
\mathcal{X}^{\ast }(E),  \label{a-b-co}
\end{equation}%
where $E=J^{1}(\mathbb{R},M^{4})$. Note that, under a change of coordinates (%
\ref{tr-rules}), the elements of the adapted bases (\ref{a-b-v}) and (\ref%
{a-b-co}) transform as classical tensors. Consequently, all subsequent
geometrical objects on the $1$-jet space $J^{1}(\mathbb{R},M^{4})$ (as
Cartan canonical connection, torsion, curvature etc.) will be described in
local adapted components.

Using a general result from \cite{Neagu-Rheon}, by direct computations, we
can give the following important geometrical result:

\begin{theorem}
The Cartan canonical $\mathring{\Gamma}$-linear connection, produced by the
rheonomic Berwald-Mo\'{o}r metric (\ref{rheon-B-M}), has the following
adapted local components:%
\begin{equation*}
C\mathring{\Gamma}=\left( \varkappa _{11}^{1},\text{ }G_{j1}^{k}=0,\text{ }%
L_{jk}^{i}=\frac{\varkappa _{11}^{1}}{3}C_{j(k)}^{i(1)},\text{ }%
C_{j(k)}^{i(1)}\right) ,
\end{equation*}%
where, if we use the notation%
\begin{equation*}
A_{jk}^{i}=\frac{2\delta _{j}^{i}+2\delta _{k}^{i}+2\delta _{jk}-8\delta
_{j}^{i}\delta _{jk}-1}{8}\text{ (no sum by }i,\text{ }j\text{ or }k\text{),}
\end{equation*}%
then%
\begin{equation*}
C_{j(k)}^{i(1)}=A_{jk}^{i}\cdot \frac{y_{1}^{i}}{y_{1}^{j}y_{1}^{k}}\text{
(no sum by }i,\text{ }j\text{ or }k\text{).}
\end{equation*}
\end{theorem}

\begin{proof}
Via the Berwald-Mo\'{o}r derivative operators (\ref{a-b-v}) and (\ref{a-b-co}%
), we use the general formulas which give the adapted components of the
Cartan canonical connection, namely \cite{Neagu-Rheon}%
\begin{equation*}
G_{j1}^{k}=\frac{g^{km}}{2}\frac{\delta g_{mj}}{\delta t},\quad L_{jk}^{i}=%
\frac{g^{im}}{2}\left( \frac{\delta g_{jm}}{\delta x^{k}}+\frac{\delta g_{km}%
}{\delta x^{j}}-\frac{\delta g_{jk}}{\delta x^{m}}\right) ,
\end{equation*}%
\begin{equation*}
C_{j(k)}^{i(1)}=\frac{g^{im}}{2}\left( \frac{\partial g_{jm}}{\partial
y_{1}^{k}}+\frac{\partial g_{km}}{\partial y_{1}^{j}}-\frac{\partial g_{jk}}{%
\partial y_{1}^{m}}\right) =\frac{g^{im}}{2}\frac{\partial g_{jm}}{\partial
y_{1}^{k}}.
\end{equation*}
\end{proof}

\begin{remark}
The below properties of the d-tensor $C_{j(k)}^{i(1)}$ are true (see also
the papers \cite{At-Bal-Neagu} and \cite{Mats-Shimada}):%
\begin{equation}
C_{j(k)}^{i(1)}=C_{k(j)}^{i(1)},\quad C_{j(m)}^{i(1)}y_{1}^{m}=0,\quad
C_{j(m)}^{m(1)}=0\text{ (sum by }m\text{)}.  \label{equalitie-C}
\end{equation}
\end{remark}

\begin{remark}
The coefficients $A_{ij}^{l}$ have the following values:%
\begin{equation}
A_{ij}^{l}=\left\{ 
\begin{array}{ll}
-\dfrac{1}{8}, & i\neq j\neq l\neq i\medskip \\ 
\dfrac{1}{8}, & i=j\neq l\text{ or }i=l\neq j\text{ or }j=l\neq i\medskip \\ 
-\dfrac{3}{8}, & i=j=l.%
\end{array}%
\right.  \label{A-(ijk)}
\end{equation}
\end{remark}

\begin{theorem}
The Cartan canonical connection $C\mathring{\Gamma}$ of the rheonomic
Berwald-Mo\'{o}r metric (\ref{rheon-B-M}) has \textbf{three} effective local
torsion d-tensors:%
\begin{equation*}
\begin{array}{c}
P_{(1)i(j)}^{(k)\text{ }(1)}=-\dfrac{1}{3}\varkappa
_{11}^{1}C_{i(j)}^{k(1)},\quad P_{i(j)}^{k(1)}=C_{i(j)}^{k(1)},\medskip \\ 
R_{(1)1j}^{(k)}=\dfrac{1}{3}\left[ \dfrac{d\varkappa _{11}^{1}}{dt}%
-\varkappa _{11}^{1}\varkappa _{11}^{1}\right] \delta _{j}^{k}.%
\end{array}%
\end{equation*}
\end{theorem}

\begin{proof}
A general $h$-normal $\Gamma $-linear connection on the 1-jet space $J^{1}(%
\mathbb{R},M^{4})$ is characterized by \textit{eight} effective d-tensors of
torsion (for more details, please see \cite{Neagu-Rheon}). For our Cartan
canonical connection $C\mathring{\Gamma}$ these reduce to the following 
\textit{three} (the other five cancel):%
\begin{equation*}
{P_{(1)i(j)}^{(k)\text{ }(1)}={\dfrac{\partial N_{(1)i}^{(k)}}{\partial
y_{1}^{j}}}-L_{ji}^{k}},\quad R_{(1)1j}^{(k)}={\dfrac{\delta M_{(1)1}^{(k)}}{%
\delta x^{j}}}-{\dfrac{\delta N_{(1)j}^{(k)}}{\delta t}},\quad
P_{i(j)}^{k(1)}=C_{i(j)}^{k(1)}.
\end{equation*}
\end{proof}

\begin{theorem}
The Cartan canonical connection $C\mathring{\Gamma}$ of the rheonomic
Berwald-Mo\'{o}r metric (\ref{rheon-B-M}) has \textbf{three} effective local
curvature d-tensors:%
\begin{equation*}
\begin{array}{c}
R_{ijk}^{l}=\dfrac{1}{9}\varkappa _{11}^{1}\varkappa
_{11}^{1}S_{i(j)(k)}^{l(1)(1)},\quad P_{ij(k)}^{l\text{ }(1)}=\dfrac{1}{3}%
\varkappa _{11}^{1}S_{i(j)(k)}^{l(1)(1)},\medskip \\ 
S_{i(j)(k)}^{l(1)(1)}={{\dfrac{\partial C_{i(j)}^{l(1)}}{\partial y_{1}^{k}}}%
-{\dfrac{\partial C_{i(k)}^{l(1)}}{\partial y_{1}^{j}}}%
+C_{i(j)}^{m(1)}C_{m(k)}^{l(1)}-C_{i(k)}^{m(1)}C_{m(j)}^{l(1)}.}%
\end{array}%
\end{equation*}
\end{theorem}

\begin{proof}
A general $h$-normal $\Gamma $-linear connection on the 1-jet space $J^{1}(%
\mathbb{R},M^{4})$ is characterized by \textit{five} effective d-tensors of
curvature (for more details, please see \cite{Neagu-Rheon}). For our Cartan
canonical connection $C\mathring{\Gamma}$ these reduce to the following 
\textit{three} (the other two cancel):%
\begin{equation*}
\begin{array}{l}
\medskip {R_{ijk}^{l}={\dfrac{\delta L_{ij}^{l}}{\delta x^{k}}}-{\dfrac{%
\delta L_{ik}^{l}}{\delta x^{j}}}+L_{ij}^{m}L_{mk}^{l}-L_{ik}^{m}L_{mj}^{l},}
\\ 
\medskip {P_{ij(k)}^{l\;\;(1)}={\dfrac{\partial L_{ij}^{l}}{\partial
y_{1}^{k}}}-C_{i(k)|j}^{l(1)}+C_{i(m)}^{l(1)}P_{(1)j(k)}^{(m)\;\;(1)},} \\ 
S_{i(j)(k)}^{l(1)(1)}={{\dfrac{\partial C_{i(j)}^{l(1)}}{\partial y_{1}^{k}}}%
-{\dfrac{\partial C_{i(k)}^{l(1)}}{\partial y_{1}^{j}}}%
+C_{i(j)}^{m(1)}C_{m(k)}^{l(1)}-C_{i(k)}^{m(1)}C_{m(j)}^{l(1)},}%
\end{array}%
\end{equation*}%
where%
\begin{equation*}
{C_{i(k)|j}^{l(1)}=}\frac{\delta {C_{i(k)}^{l(1)}}}{\delta x^{j}}+{%
C_{i(k)}^{m(1)}L_{mj}^{l}}-{C_{m(k)}^{l(1)}L_{ij}^{m}}-{%
C_{i(m)}^{l(1)}L_{kj}^{m}}.
\end{equation*}
\end{proof}

\begin{remark}
The curvature d-tensor $S_{i(j)(k)}^{l(1)(1)}$ has the properties%
\begin{equation*}
S_{i(j)(k)}^{l(1)(1)}+S_{i(k)(j)}^{l(1)(1)}=0,\quad S_{i(j)(j)}^{l(1)(1)}=0%
\text{ (no sum by }j\text{).}
\end{equation*}
\end{remark}

\begin{theorem}
The following expressions of the curvature d-tensor $S_{i(j)(k)}^{l(1)(1)}$
hold good:

\begin{enumerate}
\item $S_{i(j)(k)}^{l(1)(1)}=0$ for $\{i,$ $j,$ $k,$ $l\}$ distinct indices;

\item $S_{i(i)(k)}^{l(1)(1)}=-\dfrac{1}{16}\dfrac{y_{1}^{l}}{\left(
y_{1}^{i}\right) ^{2}y_{1}^{k}}$ ($i\neq k\neq l\neq i$ and no sum by $i$);

\item $S_{i(j)(i)}^{l(1)(1)}=\dfrac{1}{16}\dfrac{y_{1}^{l}}{\left(
y_{1}^{i}\right) ^{2}y_{1}^{j}}$ ($i\neq j\neq l\neq i$ and no sum by $i$);

\item $S_{i(j)(k)}^{i(1)(1)}=0$ ($i\neq j\neq k\neq i$ and no sum by $i$);

\item $S_{i(l)(k)}^{l(1)(1)}=\dfrac{1}{16y_{1}^{i}y_{1}^{k}}$ ($i\neq k\neq
l\neq i$ and no sum by $l$);

\item $S_{i(j)(l)}^{l(1)(1)}=-\dfrac{1}{16y_{1}^{i}y_{1}^{j}}$ ($i\neq j\neq
l\neq i$ and no sum by $l$);

\item $S_{i(i)(l)}^{l(1)(1)}=\dfrac{1}{8\left( y_{1}^{i}\right) ^{2}}$ ($%
i\neq l$ and no sum by $i$ or $l$);

\item $S_{i(l)(i)}^{l(1)(1)}=-\dfrac{1}{8\left( y_{1}^{i}\right) ^{2}}$ ($%
i\neq l$ and no sum by $i$ or $l$);

\item $S_{l(l)(k)}^{l(1)(1)}=0$ ($k\neq l$ and no sum by $l$);

\item $S_{l(j)(l)}^{l(1)(1)}=0$ ($j\neq l$ and no sum by $l$).
\end{enumerate}
\end{theorem}

\begin{proof}
For $j\neq k$, the expression of the curvature tensor $S_{i(j)(k)}^{l(1)(1)}$
takes the form (no sum by $i$, $j$, $k$ or $l$, but with sum by $m$) 
\begin{eqnarray*}
S_{i(j)(k)}^{l(1)(1)} &=&\left[ \frac{A_{ij}^{l}\delta _{k}^{l}}{%
y_{1}^{i}y_{1}^{j}}-\frac{A_{ik}^{l}\delta _{j}^{l}}{y_{1}^{i}y_{1}^{k}}%
\right] +\left[ \frac{A_{ik}^{l}\delta _{ij}y_{1}^{l}}{\left(
y_{1}^{i}\right) ^{2}y_{1}^{k}}-\frac{A_{ij}^{l}\delta _{ik}y_{1}^{l}}{%
\left( y_{1}^{i}\right) ^{2}y_{1}^{j}}\right] + \\
&&+\left[ A_{ij}^{m}A_{mk}^{l}-A_{ik}^{m}A_{mj}^{l}\right] \frac{y_{1}^{l}}{%
y_{1}^{i}y_{1}^{j}y_{1}^{k}},
\end{eqnarray*}%
where the coefficients $A_{ij}^{l}$ are given by the relations (\ref{A-(ijk)}%
).
\end{proof}

\section{Geometrical gravitational theory produced by the rheonomic
Berwald-Mo\'{o}r metric}

\hspace{5mm}From a physical point of view, on the 1-jet space $J^{1}(\mathbb{%
R},M^{4})$, the rheonomic Berwald-Mo\'{o}r metric (\ref{rheon-B-M}) produces
the adapted metrical d-tensor%
\begin{equation}
\mathbb{G}=h_{11}dt\otimes dt+g_{ij}dx^{i}\otimes dx^{j}+h^{11}g_{ij}\delta
y_{1}^{i}\otimes \delta y_{1}^{j},  \label{gravit-pot-B-M}
\end{equation}%
where $g_{ij}$ is given by (\ref{g-jos-(ij)}). This may be regarded as a 
\textit{"non-isotropic gravitational potential"}. In such a physical
context, the nonlinear connection $\mathring{\Gamma}$ (used in the
construction of the distinguished 1-forms $\delta y_{1}^{i}$) prescribes,
probably, a kind of \textit{\textquotedblleft interaction\textquotedblright }
between $(t)$-, $(x)$- and $(y)$-fields.

We postulate that the non-isotropic gravitational potential $\mathbb{G}$ is
governed by the \textit{geometrical Einstein equations}%
\begin{equation}
\text{Ric }\left( C\mathring{\Gamma}\right) -\frac{\text{Sc }\left( C%
\mathring{\Gamma}\right) }{2}\mathbb{G=}\mathcal{KT},
\label{Einstein-eq-global}
\end{equation}%
where Ric $\left( C\mathring{\Gamma}\right) $ is the \textit{Ricci d-tensor}
associated to the Cartan canonical connection $C\mathring{\Gamma}$ (in
Riemannian sense and using adapted bases), Sc $\left( C\mathring{\Gamma}%
\right) $ is the \textit{scalar curvature}, $\mathcal{K}$ is the \textit{%
Einstein constant} and $\mathcal{T}$ is the intrinsic \textit{stress-energy}
d-tensor of matter.

In this way, working with the adapted basis of vector fields (\ref{a-b-v}),
we can find the local geometrical Einstein equations for the rheonomic
Berwald-Mo\'{o}r metric (\ref{rheon-B-M}). Firstly, by direct computations,
we find:

\begin{proposition}
The Ricci d-tensor of the Cartan canonical connection $C\mathring{\Gamma}$
of the rheonomic Berwald-Mo\'{o}r metric (\ref{rheon-B-M}) has the following
effective local Ricci d-tensors:%
\begin{equation}
\begin{array}{l}
\medskip R_{ij}=R_{ijm}^{m}=\dfrac{1}{9}\varkappa _{11}^{1}\varkappa
_{11}^{1}S_{(i)(j)}^{(1)(1)}, \\ 
P_{i(j)}^{\text{ }(1)}=P_{(i)j}^{(1)}=P_{ij(m)}^{m\text{ }(1)}=\dfrac{1}{3}%
\varkappa _{11}^{1}S_{(i)(j)}^{(1)(1)},\medskip \\ 
S_{(i)(j)}^{(1)(1)}=S_{i(j)(m)}^{m(1)(1)}=\dfrac{7\delta _{ij}-1}{8}\dfrac{1%
}{y_{1}^{i}y_{1}^{j}}\text{ (no sum by }i\text{ or }j\text{){.}}%
\end{array}
\label{Ricci-local}
\end{equation}
\end{proposition}

\begin{remark}
The local Ricci d-tensor $S_{(i)(j)}^{(1)(1)}$ has the following expression:%
\begin{equation*}
S_{(i)(j)}^{(1)(1)}=\left\{ 
\begin{array}{ll}
-\dfrac{1}{8}\dfrac{1}{y_{1}^{i}y_{1}^{j}}, & i\neq j\medskip \\ 
\dfrac{3}{4}\dfrac{1}{\left( y_{1}^{i}\right) ^{2}}, & i=j.%
\end{array}%
\right.
\end{equation*}
\end{remark}

\begin{remark}
Using the third equality of (\ref{Ricci-local}) and the equality (\ref%
{g-sus-(jk)}), we deduce that the following equality is true (sum by $r$):%
\begin{equation}
S_{i}^{m11}\overset{def}{=}g^{mr}S_{(r)(i)}^{(1)(1)}=\frac{5-14\delta
_{i}^{m}}{4}\cdot \frac{1}{\sqrt{G_{1111}}}\cdot \frac{y_{1}^{m}}{y_{1}^{i}}%
\text{ (no sum by }i\text{ or }m\text{).}  \label{S-ridicat}
\end{equation}%
Moreover, by a direct calculation, we obtain the equalities%
\begin{equation}
\sum_{m,r=1}^{4}S_{r}^{m11}C_{i(m)}^{r(1)}=0,\quad \sum_{m=1}^{4}\frac{%
\partial S_{i}^{m11}}{\partial y_{1}^{m}}=\frac{3}{\sqrt{G_{1111}}}\dfrac{1}{%
y_{1}^{i}}.  \label{equalities-S-ridicat}
\end{equation}
\end{remark}

\begin{proposition}
The scalar curvature of the Cartan canonical connection $C\mathring{\Gamma}$
of the rheonomic Berwald-Mo\'{o}r metric (\ref{rheon-B-M}) is given by%
\begin{equation*}
\text{Sc }\left( C\mathring{\Gamma}\right) =-\frac{9h_{11}+\varkappa
_{11}^{1}\varkappa _{11}^{1}}{\sqrt{G_{1111}}}.
\end{equation*}
\end{proposition}

\begin{proof}
The general formula for the scalar curvature of a Cartan connection is (for
more details, please see \cite{Neagu-Rheon})%
\begin{equation*}
\text{Sc }\left( C\mathring{\Gamma}\right)
=g^{pq}R_{pq}+h_{11}g^{pq}S_{(p)(q)}^{(1)(1)}.
\end{equation*}
\end{proof}

Describing the global geometrical Einstein equations (\ref%
{Einstein-eq-global}) in the adapted basis of vector fields (\ref{a-b-v}),
we find the following important geometrical and physical result (for more
details, please see \cite{Neagu-Rheon}):

\begin{theorem}
The local \textbf{geometrical Einstein equations} that govern the
non-isotropic gravitational potential $\mathbb{G}$ (produced by the
rheonomic Berwald-Mo\'{o}r metric (\ref{rheon-B-M})) are given by:%
\begin{equation}
\left\{ 
\begin{array}{l}
\medskip \dfrac{\xi _{11}h_{11}}{\sqrt{G_{1111}}}=\mathcal{T}_{11} \\ 
\medskip \dfrac{\varkappa _{11}^{1}\varkappa _{11}^{1}}{9\mathcal{K}}%
S_{(i)(j)}^{(1)(1)}+\dfrac{\xi _{11}}{\sqrt{G_{1111}}}g_{ij}=\mathcal{T}_{ij}
\\ 
\dfrac{1}{\mathcal{K}}S_{(i)(j)}^{(1)(1)}+\dfrac{\xi _{11}}{\sqrt{G_{1111}}}%
h^{11}g_{ij}=\mathcal{T}_{(i)(j)}^{(1)(1)}%
\end{array}%
\right.  \label{E-1}
\end{equation}%
\medskip%
\begin{equation}
\left\{ 
\begin{array}{lll}
0=\mathcal{T}_{1i}, & 0=\mathcal{T}_{i1}, & 0=\mathcal{T}_{(i)1}^{(1)},%
\medskip \\ 
0=\mathcal{T}_{1(i)}^{\text{ }(1)}, & \dfrac{\varkappa _{11}^{1}}{3\mathcal{K%
}}S_{(i)(j)}^{(1)(1)}=\mathcal{T}_{i(j)}^{\text{ }(1)}, & \dfrac{\varkappa
_{11}^{1}}{3\mathcal{K}}S_{(i)(j)}^{(1)(1)}=\mathcal{T}_{(i)j}^{(1)},%
\end{array}%
\right.  \label{E-2}
\end{equation}%
\medskip where 
\begin{equation}
\xi _{11}=\frac{9h_{11}+\varkappa _{11}^{1}\varkappa _{11}^{1}}{2\mathcal{K}}%
.  \label{CSI}
\end{equation}
\end{theorem}

\begin{remark}
The local geometrical Einstein equations (\ref{E-1}) and (\ref{E-2}) impose
as the stress-energy d-tensor of matter $\mathcal{T}$ to be symmetrical. In
other words, the stress-energy d-tensor of matter $\mathcal{T}$ must verify
the local symmetry conditions%
\begin{equation*}
\mathcal{T}_{AB}=\mathcal{T}_{BA},\quad \forall \text{ }A,B\in \left\{ 1,%
\text{ }i,\text{ }_{(i)}^{(1)}\right\} .
\end{equation*}
\end{remark}

By direct computations, the local geometrical Einstein equations (\ref{E-1})
and (\ref{E-2}) imply the following identities of the stress-energy d-tensor
(sum by $r$):\bigskip

$\bigskip \mathcal{T}_{1}^{1}\overset{def}{=}h^{11}\mathcal{T}_{11}=\dfrac{%
\xi _{11}}{\sqrt{G_{1111}}},\quad\mathcal{T}_{1}^{m}\overset{def}{=}g^{mr}%
\mathcal{T}_{r1}=0,\quad$

$\bigskip \mathcal{T}_{(1)1}^{(m)}\overset{def}{=}h_{11}g^{mr}\mathcal{T}%
_{(r)1}^{(1)}=0,\quad\mathcal{T}_{i}^{1}\overset{def}{=}h^{11}\mathcal{T}%
_{1i}=0,$

$\bigskip\mathcal{T}_{i}^{m}\overset{def}{=}g^{mr}\mathcal{T}_{ri}=\dfrac{%
\varkappa _{11}^{1}\varkappa _{11}^{1}}{9\mathcal{K}}S_{i}^{m11}+\dfrac{\xi
_{11}}{\sqrt{G_{1111}}}\mathcal{\delta }_{i}^{m},$

$\bigskip \mathcal{T}_{(1)i}^{(m)}\overset{def}{=}h_{11}g^{mr}\mathcal{T}%
_{(r)i}^{(1)}=\dfrac{h_{11}\varkappa _{11}^{1}}{3\mathcal{K}}%
S_{i}^{m11},\quad\mathcal{T}_{\text{ \ }(i)}^{1(1)}\overset{def}{=}h^{11}%
\mathcal{T}_{1(i)}^{\text{ }(1)}=0,$

$\bigskip\mathcal{T}_{\text{ \ }(i)}^{m(1)}\overset{def}{=}g^{mr}\mathcal{T}%
_{r(i)}^{\text{ }(1)}=\dfrac{\varkappa _{11}^{1}}{3\mathcal{K}}S_{i}^{m11},$

$\bigskip\mathcal{T}_{(1)(i)}^{(m)(1)}\overset{def}{=}h_{11}g^{mr}\mathcal{T}%
_{(r)(i)}^{(1)(1)}=\dfrac{h_{11}}{\mathcal{K}}S_{i}^{m11}+\dfrac{\xi _{11}}{%
\sqrt{G_{1111}}}\mathcal{\delta }_{i}^{m},$ where the d-tensor $S_{i}^{m11}$
is given by (\ref{S-ridicat}) and $\xi _{11}$ is given by (\ref{CSI}).

\begin{corollary}
The stress-energy d-tensor of matter $\mathcal{T}$ must verify the following 
\textbf{geometrical conservation laws} (summation by $m$):%
\begin{equation*}
\left\{ 
\begin{array}{l}
\bigskip \mathcal{T}_{1/1}^{1}+\mathcal{T}_{1|m}^{m}+\mathcal{T}%
_{(1)1}^{(m)}|_{(m)}^{(1)}=\dfrac{\left( h^{11}\right) ^{2}}{8\mathcal{K}}%
\dfrac{dh_{11}}{dt}\left[ 2\dfrac{d^{2}h_{11}}{dt^{2}}-\dfrac{3}{h_{11}}%
\left( \dfrac{dh_{11}}{dt}\right) ^{2}\right] \cdot \dfrac{1}{\sqrt{G_{1111}}%
} \\ 
\bigskip \mathcal{T}_{i/1}^{1}+\mathcal{T}_{i|m}^{m}+\mathcal{T}%
_{(1)i}^{(m)}|_{(m)}^{(1)}=\dfrac{\varkappa _{11}^{1}\xi _{11}}{18}\cdot 
\dfrac{1}{\sqrt{G_{1111}}}\cdot \dfrac{1}{y_{1}^{i}} \\ 
\mathcal{T}_{\text{ \ }(i)/1}^{1(1)}+\mathcal{T}_{\text{ \ }(i)|m}^{m(1)}+%
\mathcal{T}_{(1)(i)}^{(m)(1)}|_{(m)}^{(1)}=\dfrac{\xi _{11}}{6}\cdot \dfrac{1%
}{\sqrt{G_{1111}}}\cdot \dfrac{1}{y_{1}^{i}},%
\end{array}%
\right.
\end{equation*}%
where (summation by $m$ and $r$)\bigskip

$\bigskip \mathcal{T}_{1/1}^{1}\overset{def}{=}\dfrac{\delta \mathcal{T}%
_{1}^{1}}{\delta t}+\mathcal{T}_{1}^{1}\varkappa _{11}^{1}-\mathcal{T}%
_{1}^{1}\varkappa _{11}^{1}=\dfrac{\delta \mathcal{T}_{1}^{1}}{\delta t},$

$\bigskip \mathcal{T}_{1|m}^{m}\overset{def}{=}\dfrac{\delta \mathcal{T}%
_{1}^{m}}{\delta x^{m}}+\mathcal{T}_{1}^{r}L_{rm}^{m}=\dfrac{\delta \mathcal{%
T}_{1}^{m}}{\delta x^{m}},$

$\bigskip \mathcal{T}_{(1)1}^{(m)}|_{(m)}^{(1)}\overset{def}{=}\dfrac{%
\partial \mathcal{T}_{(1)1}^{(m)}}{\partial y_{1}^{m}}+\mathcal{T}%
_{(1)1}^{(r)}C_{r(m)}^{m(1)}=\dfrac{\partial \mathcal{T}_{(1)1}^{(m)}}{%
\partial y_{1}^{m}},$

$\bigskip \mathcal{T}_{i/1}^{1}\overset{def}{=}\dfrac{\delta \mathcal{T}%
_{i}^{1}}{\delta t}+\mathcal{T}_{i}^{1}\varkappa _{11}^{1}-\mathcal{T}%
_{r}^{1}G_{i1}^{r}=\dfrac{\delta \mathcal{T}_{i}^{1}}{\delta t}+\mathcal{T}%
_{i}^{1}\varkappa _{11}^{1},$

$\bigskip \mathcal{T}_{i|m}^{m}\overset{def}{=}\dfrac{\delta \mathcal{T}%
_{i}^{m}}{\delta x^{m}}+\mathcal{T}_{i}^{r}L_{rm}^{m}-\mathcal{T}%
_{r}^{m}L_{im}^{r}=\dfrac{\varkappa _{11}^{1}}{3}\dfrac{\partial \mathcal{T}%
_{i}^{m}}{\partial y_{1}^{m}},$

$\bigskip \mathcal{T}_{(1)i}^{(m)}|_{(m)}^{(1)}\overset{def}{=}\dfrac{%
\partial \mathcal{T}_{(1)i}^{(m)}}{\partial y_{1}^{m}}+\mathcal{T}%
_{(1)i}^{(r)}C_{r(m)}^{m(1)}-\mathcal{T}_{(1)r}^{(m)}C_{i(m)}^{r(1)}=\dfrac{%
\partial \mathcal{T}_{(1)i}^{(m)}}{\partial y_{1}^{m}},$

$\bigskip \mathcal{T}_{\text{ \ }(i)/1}^{1(1)}\overset{def}{=}\dfrac{\delta 
\mathcal{T}_{\text{ \ }(i)}^{1(1)}}{\delta t}+2\mathcal{T}_{\text{ \ }%
(i)}^{1(1)}\varkappa _{11}^{1},$

$\bigskip \mathcal{T}_{\text{ \ }(i)|m}^{m(1)}\overset{def}{=}\dfrac{\delta 
\mathcal{T}_{\text{ \ }(i)}^{m(1)}}{\delta x^{m}}+\mathcal{T}_{\text{ \ }%
(i)}^{r(1)}L_{rm}^{m}-\mathcal{T}_{\text{ \ }(r)}^{m(1)}L_{im}^{r}=\dfrac{%
\varkappa _{11}^{1}}{3}\dfrac{\partial \mathcal{T}_{\text{ \ }(i)}^{m(1)}}{%
\partial y_{1}^{m}},$

$\mathcal{T}_{(1)(i)}^{(m)(1)}|_{(m)}^{(1)}\overset{def}{=}\dfrac{\partial 
\mathcal{T}_{(1)(i)}^{(m)(1)}}{\partial y_{1}^{m}}+\mathcal{T}%
_{(1)(i)}^{(r)(1)}C_{r(m)}^{m(1)}-\mathcal{T}%
_{(1)(r)}^{(m)(1)}C_{i(m)}^{r(1)}=\dfrac{\partial \mathcal{T}%
_{(1)(i)}^{(m)(1)}}{\partial y_{1}^{m}}.$
\end{corollary}

\begin{proof}
The conservation laws are provided by direct computations, using the
relations (\ref{equalitie-C}) and (\ref{equalities-S-ridicat}).
\end{proof}

\section{Some physical remarks and comments}

\subsection{On gravitational theory}

\hspace{5mm}It is known that in the classical Relativity theory of Einstein
(which characterizes the gravity in an isotropic space-time) the tensor of
matter must verify the conservation laws%
\begin{equation*}
\mathcal{T}_{i;m}^{m}=0,\quad \forall \text{ }i=\overline{1,4},
\end{equation*}%
where "$;$" means the covariant derivative produced by the Levi-Civita
connection associated to pseudo-Riemannian metric $g_{ij}(x)$ (the
gravitational potentials).

Comparatively, in our non-isotropic gravitational theory (with respect to
the rheonomic Berwald-Mo\'{o}r metric (\ref{rheon-B-M})) the conservation
laws are replaced with ($i=\overline{1,4}$)

\begin{equation*}
\begin{array}{l}
\bigskip \mathcal{T}_{1}=\dfrac{\left( h^{11}\right) ^{2}}{8\mathcal{K}}%
\dfrac{dh_{11}}{dt}\left[ 2\dfrac{d^{2}h_{11}}{dt^{2}}-\dfrac{3}{h_{11}}%
\left( \dfrac{dh_{11}}{dt}\right) ^{2}\right] \cdot \dfrac{1}{\sqrt{G_{1111}}%
} \\ 
\mathcal{T}_{i}=\dfrac{\varkappa _{11}^{1}\xi _{11}}{18}\cdot \dfrac{1}{%
\sqrt{G_{1111}}}\cdot \dfrac{1}{y_{1}^{i}},\quad\mathcal{T}_{(i)}^{(1)}=%
\dfrac{\xi _{11}}{6}\cdot \dfrac{1}{\sqrt{G_{1111}}}\cdot \dfrac{1}{y_{1}^{i}%
},%
\end{array}%
\end{equation*}

where%
\begin{equation*}
\begin{array}{l}
\medskip \mathcal{T}_{1}\overset{def}{=}\mathcal{T}_{1/1}^{1}+\mathcal{T}%
_{1|m}^{m}+\mathcal{T}_{(1)1}^{(m)}|_{(m)}^{(1)}, \\ 
\medskip \mathcal{T}_{i}\overset{def}{=}\mathcal{T}_{i/1}^{1}+\mathcal{T}%
_{i|m}^{m}+\mathcal{T}_{(1)i}^{(m)}|_{(m)}^{(1)}, \\ 
\mathcal{T}_{(i)}^{(1)}\overset{def}{=}\mathcal{T}_{\text{ \ }(i)/1}^{1(1)}+%
\mathcal{T}_{\text{ \ }(i)|m}^{m(1)}+\mathcal{T}%
_{(1)(i)}^{(m)(1)}|_{(m)}^{(1)}.%
\end{array}%
\end{equation*}

By analogy with Einstein's theory, if we impose the conditions ($\forall $ $%
i=\overline{1,4}$)%
\begin{equation*}
\left\{ 
\begin{array}{l}
\medskip \mathcal{T}_{1}=0 \\ 
\medskip \mathcal{T}_{i}=0 \\ 
\mathcal{T}_{(i)}^{(1)}=0,%
\end{array}%
\right.
\end{equation*}%
then we reach to the system of differential equations

\begin{equation}
\left\{ 
\begin{array}{l}
\dfrac{dh_{11}}{dt}\left[ 2\dfrac{d^{2}h_{11}}{dt^{2}}-\dfrac{3}{h_{11}}%
\left( \dfrac{dh_{11}}{dt}\right) ^{2}\right] =0\medskip \\ 
9h_{11}+\varkappa _{11}^{1}\varkappa _{11}^{1}=0.%
\end{array}%
\right.  \label{DEs}
\end{equation}%
Obviously, because we have $h_{11}>0$, we deduce that the DEs system (\ref%
{DEs}) has not any solution. Consequently, we always have%
\begin{equation*}
\left[ \mathcal{T}_{1}\right] ^{2}+\left[ \mathcal{T}_{i}\right] ^{2}+\left[ 
\mathcal{T}_{(i)}^{(1)}\right] ^{2}\neq 0,\quad \forall \text{ }i=\overline{%
1,4}.
\end{equation*}

In our opinion, this fact suggests that our geometrical gravitational theory
(produced by the rheonomic Berwald-Mo\'{o}r gravitational potential (\ref%
{gravit-pot-B-M})) is not suitable for media whose stress-energy
d-components are%
\begin{equation*}
\mathcal{T}_{AB}=0,\quad \forall \text{ }A,B\in \left\{ 1,\text{ }i,\text{ }%
_{(i)}^{(1)}\right\} .
\end{equation*}%
However, it is important to note that at "infinity" 
\begin{equation*}
\text{(this means that }y_{1}^{i}\rightarrow \infty ,\quad \forall \text{ }i=%
\overline{1,4}),
\end{equation*}%
our Berwald-Mo\'{o}r geometrical gravitational theory seems to be
appropriate even for media characterized by a null stress-energy d-tensor of
matter. This is because at "infinity" the stress-energy local d-tensors tend
to become zero.

\subsection{On electromagnetic theory}

\hspace{5mm}In the paper \cite{Neagu-Rheon}, a geometrical theory for
electromagnetism was also created, using only a given Lagrangian function $L$
on the 1-jet space $J^{1}(\mathbb{R},M^{4})$. In the background of the jet
relativistic rheonomic Lagrange geometry from \cite{Neagu-Rheon}, we work
with an \textit{electromagnetic distinguished }$2$\textit{-form}%
\begin{equation*}
\mathbb{F}=F_{(i)j}^{(1)}\delta y_{1}^{i}\wedge dx^{j},
\end{equation*}%
where%
\begin{equation*}
F_{(i)j}^{(1)}=\frac{h^{11}}{2}\left[
g_{jm}N_{(1)i}^{(m)}-g_{im}N_{(1)j}^{(m)}+\left(
g_{ir}L_{jm}^{r}-g_{jr}L_{im}^{r}\right) y_{1}^{m}\right] ,
\end{equation*}%
which is characterized by some natural \textit{geometrical Maxwell equations}
(for more details, please see \cite{Neagu-Rheon})

In our particular case of rheonomic Berwald-Mo\'{o}r metric (\ref{rheon-B-M}%
) and nonlinear connection (\ref{nlc-B-M}), we find the electromagnetic $2$%
-form $\mathbb{F}:=\mathbb{\mathring{F}}=0.$ Consequently, our Berwald-Mo%
\'{o}r geometrical electromagnetic theory is trivial. In our opinion, this
fact suggests that the rheonomic Berwald-Mo\'{o}r metric (\ref{rheon-B-M})
has rather strong gravitational connotations than electromagnetic ones. This
is because, in our geometrical approach, the Berwald-Mo\'{o}r
electromagnetism is trivial.

\textbf{Acknowledgements.} The author thanks Professor V. Balan for his
encouragements and useful suggestions.

\textbf{Author's address}: Mircea N{\scriptsize EAGU}

University Transilvania of Bra\c{s}ov, Faculty of Mathematics and Informatics

Department of Algebra, Geometry and Differential Equations

B-dul Iuliu Maniu, No. 50, BV 500091, Bra\c{s}ov, Romania.

\textbf{E-mail}: mircea.neagu@unitbv.ro

\textbf{Website}: http://www.2collab.com/user:mirceaneagu

\end{document}